\documentclass{amsart}
\usepackage{amssymb,eucal}

\numberwithin{equation}{section}

\theoremstyle{plain}
\newtheorem{theorem}{Theorem}
\newtheorem{THEOREM}{Theorem}

\newtheorem{proposition}[theorem]{Proposition}
\newtheorem{lemma}[theorem]{Lemma}
\newtheorem{corollary}[theorem]{Corollary}

\theoremstyle{definition}
\newtheorem{definition}[theorem]{Definition}
\newtheorem*{definition*}{Definition}

\newtheorem*{notation*}{Notation}

\theoremstyle{remark}
\newtheorem{remark}[theorem]{Remark}
\newtheorem*{remark*}{Remark}

\newtheorem*{remarks*}{Remarks}
\newtheorem{example}[theorem]{Example}
\newtheorem*{example*}{Example}

\newtheorem*{examples*}{Examples}

\numberwithin{theorem}{section}

\newcommand\iso{{\widetilde\to}}                    

\newcommand\C{\mathbb{C}}                           
\newcommand\Z{\mathbb{Z}}                           
\newcommand\Q{\mathbb{Q}}                           

\newcommand\moprm[2]{\newcommand{#1}{\mathop{\mathrm{#2}}\nolimits}}     
\newcommand\mopit[2]{\newcommand{#1}{\mathop{\mathit{#2}}\nolimits}}     
\newcommand\mopbf[2]{\newcommand{#1}{\mathop{\mathbf{#2}}\nolimits}}     

\moprm{\Hom}{Hom}                                   
\moprm{\End}{End}                                   
\moprm{\Aut}{Aut}                                   
\moprm{\Iso}{Iso}                                   
\mopit{\HOM}{\cH om}                                
\mopit{\END}{\cE nd}                                
\mopit{\AUT}{\cA ut}                                
\mopit{\ISO}{\cI so}                                

\moprm{\Ext}{Ext}                                   
\moprm{\Tor}{Tor}                                   
\mopit{\EXT}{\cE xt}                                
\mopit{\TOR}{\cT or}                                

\moprm{\ob}{Ob}                                     

\moprm{\tr}{tr}                                     
\moprm{\rk}{rk}                                     
\moprm{\ad}{ad}                                     
\moprm{\Ad}{Ad}                                     
\moprm{\id}{id}                                     
\moprm{\supp}{supp}                                 
\moprm{\chr}{char}                                  
\moprm{\codim}{codim}                               
\moprm{\res}{res}                                   

\moprm{\im}{im}										
\moprm{\coim}{coim}									
\moprm{\coker}{coker}								

\moprm{\spec}{Spec}                                 
\moprm{\spf}{Spf}                                   

\mopbf{\gm}{G_m}                                    
\mopbf{\ga}{G_a}                                    

\newcommand\p[1]{{\mathbb{P}^{#1}}}                 
\newcommand\A[1]{{\mathbb{A}^{#1}}}					

\moprm{\Gal}{Gal}									

\newcommand{\cA}{{\mathcal A}}

\newcommand{\cD}{{\mathcal D}}
\newcommand{\cE}{{\mathcal E}}

\newcommand{\cH}{{\mathcal H}}
\newcommand{\cI}{{\mathcal I}}

\newcommand{\cK}{{\mathcal K}}
\newcommand{\cL}{{\mathcal L}}
\newcommand{\cM}{{\mathcal M}}

\newcommand{\cO}{{\mathcal O}}

\newcommand{\cR}{{\mathcal R}}

\newcommand{\cT}{{\mathcal T}}

\newcommand{\cV}{{\mathcal V}}
\newcommand{\cW}{{\mathcal W}}

\newcommand{\tK}{{\widetilde K}}

\newcommand{\fM}{{\mathfrak M}}

\mopit{\Hol}{\cH ol}
\mopit{\Mod}{\cM od}
\newcommand\DHol[1]{\Hol(\cD_{#1})}                        

\newcommand\kk{\Bbbk}
\newcommand\D{\cD}                                         
\newcommand\dz{{\frac{d}{d z}}}
\newcommand\ft[1]{{#1}^\wedge}
\moprm{\rig}{rig}
\moprm{\ord}{ord}
\moprm{\diag}{diag}
\moprm{\slope}{slope}
\moprm{\irreg}{irreg}
\moprm{\Char}{char}
\moprm{\gl}{\mathfrak{gl}}

\newcommand\Four{\mathfrak{F}}
\newcommand\four{F}
\newcommand\Rad{\cR}

\begin{document}
\title{Rigid irregular connections on $\p1$}
\author{D.~Arinkin}
\address{Department of Mathematics\\University of North Carolina\\Chapel Hill, NC\\USA}

\begin{abstract}
N.~Katz's middle convolution algorithm provides a description of rigid connections on $\p1$ with regular singularities. We extend the algorithm by adding
the Fourier transform to it. The extended algorithm provides a description of rigid connections with arbitrary singularities.
\end{abstract}
\maketitle

\section{Introduction}

In \cite{Ka}, N.~Katz suggested a new method of studying a local system $\cL$ on an open subset $U=\p1-\{x_1,\dots,x_n\}$: the middle convolution algorithm. He defined the middle convolution of
local systems on $\p1$, and showed that for a Kummer local system $\cK$, the operation of middle convolution with $\cK$ is invertible:
$$\cL=(\cL\star_{mid}\cK)\star_{mid}\cK^{-1}.$$
Here $\star_{mid}$ is the middle convolution. Usually, $\rk(\cL\star_{mid}\cK)\ne\rk\cL$.

To apply Katz's middle convolution algorithm to $\cL$, one looks for a rank one local system
$\ell_1$ and a Kummer local system $\cK_1$ such that the middle convolution
$$\cL_1=(\cL\otimes\ell_1)\star_{mid}\cK_1$$
has strictly smaller rank. The process is repeated until one arrives at
the local system $\cL_k$ whose rank can no longer be decreased by this operation. Note that
$\cL$ can be reconstructed from a smaller rank local system $\cL_k$ and the sequence of rank one local systems
$\{(\ell_1,\cK_1),\dots,(\ell_k,\cK_k)\}$ used in the algorithm. 
The isomorphism class of $\cL$ is encoded by the isomorphism class of $\cL_k$ and the 
monodromies of $\ell_i$'s and $\cK_i$'s.

Katz applied the algorithm to rigid local system (a local system is \emph{rigid} if it is determined up to isomorphism by the conjugacy classes of its local monodromies).  He showed that any rigid irreducible local system $\cL$ is reduced by the algorithm to a rank one system $\cL_k$. This describes rigid irreducible
local systems using collections of numbers (the monodromies of $\cL_k$, $\ell_i$'s, and 
$\cK_i$'s).
Since then, the algorithm found numerous applications to both rigid and non-rigid local systems, see \cite{Si} for a summary.

Katz's middle convolution algorithm applies to the following `flavors' of local systems:

\begin{itemize}
\item Representations of the fundamental group of $\pi_1(\mathbb{P}_\mathbb{C}^1-\{x_1,\dots,x_n\})$ (`Betti flavor');

\item Tamely ramified $l$-adic local systems on $\mathbb{P}_{\kk}^1-\{x_1,\dots,x_n\}$ for
any field $\kk$, where $l$ is a prime distinct from $\Char(\kk)$;

\item Vector bundles with connections on $\mathbb{P}_{\kk}^1-\{x_1,\dots,x_n\}$ with
regular singularities at the punctures $x_1,\dots,x_n$ for any field $\kk$ of characteristic
zero (`de Rham flavor'). In classical language, one works with linear ordinary differential equations with Fuchsian singularities.
\end{itemize}

In this paper, we take the de Rham point of view. We extend the middle convolution algorithm
to connections with irregular singularities by using two operations: the middle
convolution and the Fourier transform. We call this extension \emph{irregular Katz's algorithm}. It is described in Section~\ref{sc:proof}; here is a short summary.

For a bundle with connection $\cL$ on an open set $U\subset\A1$, denote its Fourier transform by $\ft\cL$.
The Fourier transform is a bundle with connection $\ft\cL$ on an open subset $\ft{U}\subset\A1$; usually 
$\ft{U}\ne U$ and $\rk \ft\cL\ne\rk \cL$. 

On the each step of irregular Katz's algorithm, we try to lower the rank of $\cL$ by one of the following two operations:

\begin{enumerate}
\item\label{it:R} Replacing $\cL$ with the middle convolution $$\cL_1=(\cL\otimes\ell)\star_{mid}\cK^\lambda$$ for appropriate
choices of a line bundle with connection $\ell$ and a Kummer local system $\cK^\lambda$.  

\item\label{it:F} Replacing $\cL$ with the Fourier transform $$\cL_1=\ft{(\cL\otimes\ell)}$$ for appropriate
choice of a line bundle with connection $\ell$ and the choice of the infinity 
$\infty\in\p1$. (The point $\infty\in\p1$ plays a special
role in the Fourier transform, and we use it as a parameter in the operation.)
\end{enumerate}

Both operations \eqref{it:R} and \eqref{it:F} are invertible, so $\cL$ is determined up to isomorphism by $\cL_1$ and the numerical parameters used in the operation. We repeat this
procedure to decrease the rank of $\cL$ as much as possible.

In this paper, we work with rigid irreducible bundles with connections $\cL$. By definition,
irreducible $\cL$ is \emph{rigid} if it is determined up to isomorphism by the formal types of its singularities. The main result is that irregular Katz's algorithm always reduces such $\cL$
to a rank one bundle with connection; this yields a recursive description of irreducible rigid connections.

Our result answers the question posed by N.~Katz in \cite[p.~10]{Ka}. Also, in the introduction to \cite{BE}, S.~Bloch and
H.~Esnault express hope that their result (see Theorem~\ref{th:fourierrigid}) can be used to classify rigid connections
with irregular singularities; our paper provides such a classification.

We hope that irregular Katz's algorithm has other applications. Two examples
are discussed in Sections~\ref{sc:DS},~\ref{sc:rig0}.

\subsection{Acknowledgments}
I am very grateful to P.~Belkale for igniting my interest in Katz's middle convolution algorithm
and to V.~Drinfeld for sharing his views on the middle convolution. 

When I gave a talk on this subject at the Institute for Advanced Study, I learned that the 
extension of Katz's algorithm is also presented in a letter by P.~Deligne to N.~Katz. 
I would like to thank P.~Deligne for a copy of the letter. 

I discussed these results with many mathematicians. Besides those mentioned above, I would like to thank S.~Bloch, P.~Boalch, and A.~Varchenko. 
I am also grateful to the referee for useful comments.

\section{Main results}
Fix the ground field $\kk$, which is algebraically closed of characteristic zero.

\subsection{Connections and rigidity}

By definition, a bundle with connection on a nonempty open set $U\subset\p1$ is a pair
$\cL=(E_\cL,\nabla_\cL)$, where $E_\cL$ is a vector bundle on $U$ and the connection
$\nabla_\cL:E_\cL\to E_\cL\otimes\Omega_U$ is a $\kk$-linear map that satisfies the Leibniz identity
$$\nabla_\cL(fs)=f\nabla_\cL(s)+s\otimes df\qquad(f\in\cO_U,s\in E_\cL).$$
We simply say that $\cL$ is a connection on $U$.  
This paper uses the `de Rham' point of view, so the terms 
`local system on $U$' and `connection on $U$' are interchangeable. 

For a closed point $x\in\p1$, let $K_x$ denote the ring of formal Laurent series at
$x$. A choice of local coordinate $z$ identifies $K_x$ with $\kk((z))$. Let $\DHol{K_x}$
be the category of holonomic $\D$-modules on the punctured formal neighborhood of $x$.
 Explicitly, objects of $\DHol{K_x}$ are pairs $\cV=(E_\cV,\nabla_\cV)$, where $E_\cV$ is a finite-dimensional vector space over $K_x$ and $$\nabla_\cV:E_\cV\to E_\cV\otimes\Omega_{K_x}=\cV dz$$
is a $\kk$-linear map satisfying the Leibniz identity. We sometimes call $\cV$ a connection
on the punctured formal disk.

For two connections $\cL,\cL'$ on open set $U\subset\p1$, 
we denote by $\HOM(\cL,\cL')$ the local system of morphisms from $\cL$ to $\cL'$; equivalently,
$\HOM(\cL,\cL')=\cL'\otimes\cL^\vee$. By definition, $\END(\cL)=\HOM(\cL,\cL)$. We use similar
notation $\HOM(\cV,\cW)$, $\END(\cV)$ for $\cV,\cW\in\DHol{K_x}$, $x\in\p1$.

A connection $\cL$ on $U$ yields an object $\Psi_x(\cL)\in\DHol{K_x}$ for 
any $x\in\p1$. 
Essentially, $\Psi_x(\cL)$ is the restriction of $\cL$ to the punctured formal neighborhood of $x$:
$\Psi_x(\cL)=\cL\otimes K_x$. One can view $\Psi_x(\cL)$ as the nearby cycles of $\cL$.

\begin{definition} The \emph{formal type} $[\Psi_x(\cL)]$ of $\cL$ \emph{at} $x\in\p1$ is the isomorphism class of $\Psi_x(\cL)$. The \emph{formal type} of $\cL$ is the collection 
$$\{[\Psi_x(\cL)]\}_{x\in\p1}.$$
\end{definition} 
 
\begin{remark*} For $x\in U$, the restriction $\Psi_x(\cL)$ is trivial, so 
$[\Psi_x(\cL)]$ is given by $\rk(\cL)$. Therefore, the formal type of $\cL$ is determined by
the collection $$\{[\Psi_x(\cL)]\}_{x\in\p1-U}$$
(excluding the case when $U=\p1$ and $\cL$ is trivial). 
\end{remark*}

\begin{definition} A connection $\cL$ on $U$ is \emph{rigid} if $\cL$ is determined
by its formal type up to isomorphism: any bundle with connection
$\cL'$ on $U$ such that $\Psi_x(\cL)\simeq\Psi_x(\cL')$ for all $x\in\p1$ is isomorphic to $\cL$.
\label{df:rigid}
\end{definition}

\begin{example} Suppose $\kk=\C$ and that $\cL$ has regular singularities; that is, $\cL=(E_\cL,\nabla_\cL)$ can be extended to a vector
bundle $\overline E_\cL$ on $\p1$ equipped with a connection $\overline\nabla_\cL$ that has first-order poles:
$$\overline\nabla_\cL:\overline E_\cL\to\overline E_\cL\otimes\Omega_\p1(x_1+\dots+x_n),$$
where $\{x_1,\dots,x_n\}=\p1-U$. Then $[\Psi_x(\cL)]$ is determined by the monodromy
of $\nabla_\cL$ around $x$. The monodromy is defined up to conjugation, for
instance, it can be given in the Jordan form. Therefore, for regular connections on 
${\mathbb P}_\C^1$, Definition~\ref{df:rigid} reduces to the notion of rigidity given in the introduction.
\end{example}

\subsection{Fourier transform}
Recall the Fourier transform for $\D_\A1$-modules.  We can identify $\D_\A1$-modules with
modules over the algebra of polynomial differential operators 
$\kk\left\langle z,\dz\right\rangle$ (\emph{the Weyl algebra}). Here $z$ is
the coordinate on $\A1$.
Consider the Fourier automorphism 
$$\four:\kk\left\langle z,\dz\right\rangle\to\kk\left\langle z,\dz\right\rangle:\qquad
\four(z)=-\dz,\quad\four\left(\dz\right)=z.$$ 
It yields an autoequivalence of the category of $\kk\langle z,\dz\rangle$-modules (\emph{the Fourier transform})
$$M\mapsto\Four(M),$$ where $\Four(M)$ is isomorphic to $M$ as a vector space, but 
$\kk\langle z,\dz\rangle$ acts on it
through $\four$.

Now let $\cL$ be a connection on an open set $U\subset\A1$. 
Assume that $\cL$ is irreducible. Viewing $\cL$ as a $\D_U$-module, we obtain the Goresky-MacPherson extension $j_{!*}\cL$, where $j:U\hookrightarrow\A1$ is the open embedding. $j_{!*}\cL$ is an irreducible $\D_\A1$-module, therefore its Fourier transform
$\Four(j_{!*}\cL)$ is also irreducible. 

The $\D_\A1$-module $\Four(j_{!*}\cL)$ is smooth on a non-empty open subset $\ft{U}\subset\p1$;
that is, it gives a connection $\ft\cL$ on $\ft{U}$. 
Let us exclude the (essentially trivial) case when $\cL$ has rank one and its only singularity is a second order pole at infinity: in this case, $\Four(j_{!*}\cL)$ 
is supported at a single point, and $\ft\cL=0$. Then $$\Four(j_{!*}\cL)=\ft{j}_{!*}(\ft\cL),\quad\ft{j}:\ft{U}\hookrightarrow\A1$$
and $\ft\cL$ is an irreducible connection on $\ft{U}$. When it does not cause confusion, we call $\ft{\cL}$ the Fourier transform of $\cL$.

Fourier transform preserves rigidity. In $l$-adic settings, this was proved by N.~Katz using
the local Fourier transform constructed by G.~Laumon in \cite{La}. 
In the settings of bundles with connections, the local Fourier transform was constructed by S.~Bloch and H.~Esnault in \cite{BE}. 

\begin{theorem}[S.~Bloch, H.~Esnault] \label{th:fourierrigid}
Suppose $\cL$ is irreducible and rigid. Then so is its Fourier transform $\ft\cL$. \qed
\end{theorem}

\subsection{Middle convolution}
Fix $\lambda\in\kk-\Z$. The corresponding \emph{Kummer local system} is 
$$\cK^\lambda=\left(\cO_{\A1-\{0\}},d+\lambda\frac{dz}{z}\right).$$ 
Up to isomorphism, $\cK^\lambda$ depends only on the image of $\lambda$ in $\kk/\Z$.

Let $\cL$ be an irreducible connection on an open subset $U\subset\p1$. Shrinking $U$ if necessary, we may assume that $U\subset\A1$. We then define the middle convolution 
$\cL\star_{mid}\cK^\lambda$
to be the inverse Fourier transform of $\ft\cL\otimes\cK^{-\lambda}$:
\begin{equation}
\ft{(\cL\star_{mid}\cK^\lambda)}=\ft\cL\otimes\cK^{-\lambda}.\label{eq:mid}
\end{equation}
This definition uses the isomorphism $\ft{(\cK^\lambda)}=\cK^{-\lambda}$ to rewrite
the convolution as a tensor product. 

\begin{remark*} \cite{Ka} contains a direct definition of middle convolution that does not
use the Fourier transform. The equivalence between this definition and \eqref{eq:mid} is
\cite[Proposition~2.10.5]{Ka}. Another approach to middle convolution \eqref{eq:mid} is sketched in Section~\ref{sc:mid}. 
\end{remark*}

Let us make \eqref{eq:mid} explicit. Consider again the Fourier transform 
$\Four(j_{!*}\cL)$. It is a $\D_\A1$-module. The tensor product 
$\Four(j_{!*}\cL)\otimes\cK^{-\lambda}$ is a $\D$-module on $\A1-\{0\}$. Consider  $j_0:\A1-\{0\}\hookrightarrow\A1$, and take
$$\Four^{-1}(j_{0,!*}(\Four(j_{!*}\cL)\otimes\cK^{-\lambda})).$$
This is a $\D_\A1$-module, and $\cL\star_{mid}\cK^\lambda$ is the corresponding connection.

Again, exclude the essentially trivial case when $\cL$ is a rank one connection which is either trivial or has two simple poles at $\infty$ and some point $x\in\A1$ with residues 
equal to $\lambda$ and $-\lambda$, respectively. Then $\cL\star_{mid}\cK^\lambda$ is
again an irreducible connection. Theorem~\ref{th:fourierrigid} immediately implies that
$\cL$ is rigid if and only if so is $\cL\star_{mid}\cK^\lambda$. Clearly,
$$\cL\simeq(\cL\star_{mid}\cK^\lambda)\star_{mid}\cK^{-\lambda}.$$

\subsection{Main theorem} 
Here is the main result of this paper, proved in Section~\ref{sc:proof}.

\begin{THEOREM}
Let $\cL$ be a connection on an open subset $U\subset\p1$. Suppose $\cL$ is irreducible and rigid, and that $\rk(\cL)>1$. Then at least one of the following conditions hold:

\begin{enumerate}
\item \label{cs:R}
For appropriate $\lambda\not\in\Z$ and a rank one connection $\ell$ on $U-\{\infty\}$, 
the middle convolution
$\HOM(\ell,\cL)\star_{mid}\cK^\lambda$ has rank smaller than $\cL$.

\item \label{cs:F}
For appropriate choice of $\infty\in\p1-U$ and a rank one connection $\ell$ on $U$,
the Fourier transform of $\HOM(\ell,\cL)$ has rank smaller than $\cL$.
\end{enumerate}
\label{th:maintheorem}
\end{THEOREM}

\begin{remarks*} The Fourier transform can be thought of as a middle convolution on the multiplicative group. In this way, both cases~\eqref{cs:R},~\eqref{cs:F} involve middle convolution.

In case~\eqref{cs:F}, we use Fourier transform corresponding to some choice
of $\infty\in\p1$; the choice depends on $\cL$. Equivalently, we might fix $\infty\in\p1$, and use M\"obius transformations to shift the connection $\cL$. We then reformulate \eqref{cs:F}
as follows:
\begin{enumerate}
\item[(\ref{cs:F}')] \sl{There is a rank one connection $\ell$ on $U$ and a M\"obius
transformation $\phi:\p1\iso\p1$ such that $$\rk\ft{\left(\phi^*\HOM(\ell,\cL)\right)}<\rk(\cL).$$}
\end{enumerate}
\end{remarks*}

Theorem~\ref{th:maintheorem} yields a connection $\cL_1$ given by one of the two rules
$$\cL_1=\begin{cases}(\HOM(\ell,\cL)\star_{mid}\cK^\lambda,\quad&\text{case \eqref{cs:R} of Theorem~\ref{th:maintheorem}}\\
\ft{\HOM(\ell,\cL)},\quad&\text{case \eqref{cs:F} of Theorem~\ref{th:maintheorem}}
\end{cases}
$$
such that $\rk(\cL_1)<\rk\cL$. Note that $\cL_1$ is again irreducible and rigid (by Theorem~\ref{th:fourierrigid}), so either $\rk(\cL_1)=1$, or $\rk(\cL_1)$ can be decreased further by
Theorem~\ref{th:maintheorem}. Iterating, we eventually get to a rank one connection. This
proves the following claim.

\begin{corollary} Any rigid connection $\cL$ on open subset $U\subset\p1$ can be reduced to the trivial connection $(\cO,d)$ by iterating the following three operations:
\begin{itemize}
\item Tensor multiplication by a rank one connection $\ell$: $\cL\mapsto\cL\otimes\ell$;
\item Change of variable by a M\"obius transformation $\phi$: $\cL\mapsto\phi^*\cL$;
\item Fourier transform: $\cL\mapsto\ft{\cL}$.
\end{itemize}
Of course, $\cL$ can also be obtained from $(\cO,d)$ by these operations.
\qed
\end{corollary}

\section{Connections and Fourier transform}

In this section, we remind the necessary statements about bundles with connections. 

\subsection{Euler-Poincar\'e formula}
Fix a point $x\in\p1$. For $\cV\in\DHol{K_x}$, we denote by $\irreg(\cV)$ the irregularity 
of $\cV$ and by
$$\slope(\cV)=\frac{\irreg(\cV)}{\rk(\cV)}$$
the slope of $\cV$. It is also convenient to introduce the following quantity:
\begin{equation}
\delta(\cV)=\irreg(\cV)+\rk(\cV)-\rk(\cV^{hor}),\label{eq:delta}
\end{equation}
where $\cV^{hor}\subset\cV$ is the maximal subbundle on which the connection is 
trivial. In other words, 
$$\rk(\cV^{hor})=\dim_\kk H^0_{dR}(K_x,\cV),\qquad H^0_{dR}(K_x,\cV)=\ker(\nabla_\cV:E_\cV\to E_\cV\otimes\Omega).$$

Let $\cL$ be a connection on an open subset $U\subset\p1$. Consider the $\D_\p1$-module $j_{!*}(\cL)$ for $j:U\hookrightarrow\p1$.
Denote by $H^i_{dR}(\p1,j_{!*}(\cL))$ its de Rham cohomology groups and by
$$\chi(j_{!*}(\cL))=\sum_{i=0}^2(-1)^i\dim H^i_{dR}(\p1,j_{!*}(\cL))$$
its Euler characteristic.

We need the Euler-Poincar\'e formula for the Euler characteristic:
\begin{proposition} Let $\cL$ and $j:U\hookrightarrow\p1$ be as above. Then
$$\chi(j_{!*}(\cL))=2\rk(\cL)-\sum_{x\in\p1-U}\delta(\Psi_x(\cL)).$$ \qed
\label{pp:EulerPoincare}
\end{proposition}

\subsection{Rigidity index} \label{sc:rigidity}
\begin{definition}
For $\cL$ as above, the \emph{rigidity index} of $\cL$ is given by
$$\rig(\cL)=\chi(j_{!*}\END(\cL)).$$
\end{definition}

\begin{remark} It is well known that $\rig(\cL)$ is even. Indeed, by the Verdier duality 
the vector spaces $H^0_{dR}(\p1,j_{!*}(\END(\cL)))$ and $H^2_{dR}(\p1,j_{!*}(\END(\cL))$ are dual (and therefore have equal dimension), while $H^1_{dR}(\p1,j_{!*}(\END(\cL)))$ carries a symplectic
form (and therefore has even dimension).
\end{remark}

The following statement is an extension of \cite[Theorem~1.1.2]{Ka} to the case of irregular singularities. 
\begin{proposition}[{\cite[Theorem~4.7, Theorem~4.10]{BE}}]
An irreducible connection $\cL$ is rigid if and only if $\rig(\cL)=2$. \qed
\end{proposition}
\begin{remark*} For irreducible $\cL$, we have $$\rig(\cL)=2-\dim H^1_{dR}(\p1,j_{!*}(\END(\cL)).$$ Therefore, two is the largest possible value of $\rig(\cL)$
(and the only positive value). In \cite{Ka}, local systems satisfying $\rig(\cL)=2$ are called
\emph{cohomologically rigid}, while those satisfying Definition \ref{df:rigid} are \emph{physically rigid}. 
\end{remark*}

\subsection{Rank of the Fourier transform}
Suppose now that $\cL$ is a connection on an open subset $U\subset\A1$. Consider
the Fourier transform $\Four(j_{!*}(\cL))$ for $j:U\hookrightarrow\A1$. We want to find the
(generic) rank of the Fourier transform, that is, $\rk\ft\cL$.

\begin{proposition}[{\cite[Proposition V.1.5]{Mal}}]\label{pp:rankfourier}
Denote by $\Psi_\infty(\cL)^{>1}\subset\Psi_\infty(\cL)$ the maximal submodule of
$\cL$ whose irreducible components all have slopes greater than one. Then
$$\rk(\ft\cL)=\sum_{x\in\A1-U}\delta(\Psi_x(\cL))+
\irreg(\Psi_\infty(\cL)^{>1})-\rk(\Psi_\infty(\cL)^{>1}).$$\qed
\end{proposition}

Similarly, we have a formula for the rank of the middle convolution. In the case of regular
singularities, this is \cite[Corollary~3.3.7]{Ka} (in $l$-adic settings).

\begin{proposition}\label{pp:rankradon} Denote by $\cK_\infty^\lambda\in\DHol{K_\infty}$
the `Kummer local system at infinity' given by $$\cK_\infty^\lambda=(K_\infty,d+\lambda\frac{d\zeta}{\zeta}),$$ 
where $\zeta$ is a local coordinate at $\infty\in\p1$. Note that 
the residue of $\cK^\lambda_\infty$ is $\lambda$, so $\cK^\lambda_\infty\simeq\Psi_\infty(\cK^{-\lambda})$. Then
$$\rk(\cL\star_{mid}\cK^\lambda)=\sum_{x\in\A1-U}\delta(\Psi_x(\cL))+
\delta(\Psi_\infty(\cL)\otimes\cK^{-\lambda}_\infty)-\rk\cL.$$
\end{proposition}
\begin{proof}
For any point $x\in U$, the fiber of $\cL\star_{mid}\cK^\lambda$ over $x$ equals
\begin{equation}
(\cL\star_{mid}\cK^\lambda)_x=H^1_{dR}(\p1,j_{!*}(\cL\otimes s_x^*(\cK^\lambda))),\label{eq:conv}
\end{equation}
where $j:U\hookrightarrow\p1$ is the embedding and $s_x:\p1\iso\p1$ is given by $y\mapsto x-y$
(so $s_x^*\cK^\lambda$ has regular singularities at $x$ and $\infty$). The formula 
\eqref{eq:conv} is essentially \cite[Corollary 2.8.5]{Ka}. It is easy to prove if one views
the middle convolution as an integral transform, as in Section~\ref{sc:mid}.

Clearly, $H^i(\p1,j_{!*}(\cL\otimes s_x^*(\cK^\lambda)))=0$ for $i=0,2$, so
$$\rk(\cL\star_{mid}\cK^\lambda)=-\chi(j_{!*}(\cL\otimes s_x^*(\cK^\lambda))).$$
It remains to apply Proposition~\ref{pp:EulerPoincare}.
\end{proof}

\section{Proof of Theorem~\ref{th:maintheorem}}\label{sc:proof}

\subsection{Outline of proof}
For every singular point $x\in\p1-U$, consider the formal type $\Psi_x(\cL)\in\DHol{K_x}$ of $\cL$
at $x$. Choose an irreducible connection
$\cV_x\in\DHol{K_x}$ that minimizes $$\frac{\delta(\HOM(\cV_x,\Psi_x(\cL)))}{\rk(\cV_x)}$$ 
(this choice is described in Corollary~\ref{co:maxslope}).

\begin{example} \label{ex:regularmultipliers}
Suppose $\cL$ has regular singularities. Then $\rk(\cV_x)=1$, and $\cV_x$ has regular singularities. Therefore, $\cV_x\simeq(K_x,d+\lambda\frac{dz}{z})$, where $\lambda$ is chosen so as to maximize $\dim\Hom(\cV_x,\Psi_x(\cL))$. 
Here $z$ is a local coordinate at $x$. Explicitly, we can write $\Psi_x(\cL)\simeq(K_x^r,d+R\frac{dz}{z})$, where $R$ is an $r\times r$ matrix with constant 
coefficients such that no two eigenvalues of $R$ differ by a non-zero integer. 
Then $\lambda$ is the eigenvalue of $R$ with maximal geometric multiplicity (that is, the eigenspace of $\lambda$ has maximal dimension).
 
If $\kk=\C$, we can simply say that $\cV_x$ is given by the eigenvalue of the monodromy of $\cL$ with maximal geometric multiplicity. 
\end{example} 

\emph{Case I:} Suppose $\rk(\cV_x)=1$ for all $x\in\p1-U$. (By Example~\ref{ex:regularmultipliers}, this is true if $\cL$ has regular singularities, so this is the only case appearing in the middle convolution algorithm of \cite{Ka}.)
It makes sense to talk about $\res(\cV_x)\in\kk/\Z$.

\emph{Case Ia:} Suppose $\sum_x\res(\cV_x)\in\Z$. Then one can find a local system 
$\ell$ on $U$ such that $\Psi_x(\ell)\simeq \cV_x$. One can easily see from the Euler-Poincar\'e formula that either $\Hom(\cL,\ell)$ or $\Hom(\ell,\cL)$ is non-zero (Proposition~\ref{pp:Case1a}). 
Since $\cL$ is irreducible, this implies $\cL\simeq\ell$, so $\rk(\cL)=1$, which contradicts the
assumptions.

\emph{Case Ib:} Suppose $\lambda=\sum_x\res(\cV_x)\not\in\Z$. Shrinking $U$ if necessary, we may assume that $\infty\not\in U$. Then there is a rank one connection
$\ell$ on $U$ such that
\begin{equation}
\Psi_x(\ell)\simeq\begin{cases} \cV_x, x\in\A1-U\cr \cV_\infty\otimes\cK^{-\lambda}_\infty,x=\infty,\end{cases}
\label{eq:twistedell}
\end{equation}
for $\cK^\lambda_\infty$ as in Proposition~\ref{pp:rankradon}.
It follows from Proposition~\ref{pp:rankradon}
that $\cL$ satisfies Theorem~\ref{th:maintheorem}\eqref{cs:R} 
for this $\ell$ (Proposition~\ref{pp:Case1b}).

\emph{Case II:} Suppose $\rk(\cV_x)>1$ for some $x\in\p1-U$. We show (Proposition~\ref{pp:therecanbeonlyone}) that
there is unique $x$ with this property. Choose it as $\infty$. Then choose $\ell$ to be 
a rank one connection on $U$ that satisfies $\Psi_x(\ell)\simeq \cV_x$ for
$x\in\A1-U$, and such that  
$$\HOM(\Psi_\infty(\ell),\cV_\infty)\in\DHol{K_\infty}$$ 
has non-integer slope. It follows from Proposition~\ref{pp:rankfourier} that 
$\cL$ satisfies Theorem~\ref{th:maintheorem}\eqref{cs:F} for this $\ell$ (Proposition~\ref{pp:Case2}). \qed

\begin{remark} \label{rm:approximation}
Let us discuss the condition on $\ell$ in Case II. Choose
$\ell_\infty\in\DHol{K_\infty}$ to be a rank one connection which is the
`best approximation' of $\cV_\infty$ in the sense that it minimizes 
$$\slope\HOM(\ell_\infty,\cV_\infty).$$
It is clear that the minimal slope is not an integer. Note that $\ell_\infty$ is not unique;
in particular, it can be tensored by a rank one bundle with regular connection. This means
that $\res(\ell_\infty)$ is unrestricted, and so $\ell_\infty$ can be chosen so that
$$\res\ell_\infty+\sum_{x\in\A1-U}\res\cV_x\in\Z.$$
We can then find $\ell$ with $\Psi_x(\ell)\simeq\cV_x$ for $x\in\A1-U$ and 
$\Psi_\infty(\ell)\simeq\ell_\infty$.

More explicitly,  let $\zeta$ 
be a local coordinate at $\infty$, and $r=\rk(\cV_\infty)$. We apply the well-known
description of connections on a punctured formal disk (see for instance
\cite[Theorem III.1.2]{Mal}). Since $\cV_\infty$ is irreducible, there exists
a ramified extension $$\kk((\zeta^{1/r}))\supset\kk((\zeta))=K_\infty$$ and a differential
form $\mu\in\kk((\zeta^{1/r}))d\zeta$ such that $$\cV_\infty\simeq(\kk((\zeta^{1/r})),d+\mu).$$
Choose a differential form $\mu_\ell\in\kk((\zeta))d\zeta$ as a `best approximation' of
$\mu$ in the sense that the leading term of $\mu-\mu_\ell$ is a fractional power of $\zeta$. 
Then take $$\ell_\infty=(K_\infty,d+\mu_\ell).$$
\end{remark}
 
\subsection{Details of the proof: Case I}
Let us fill in the gaps in the above outline. We start with some local calculations. Fix
$x\in\p1$. 

Recall that for $\cV\in\DHol{K_x}$, $\delta(\cV)$ is defined by \eqref{eq:delta}.
It is obvious that $\delta$ is semiadditive:
\begin{lemma} For a short exact sequence $$0\to \cV_1\to \cV\to \cV_2\to 0$$ in $\DHol{K_x}$, we have
$$\delta(\cV)\ge\delta(\cV_1)+\delta(\cV_2).$$
\end{lemma}
\begin{proof}
Indeed, 
\begin{align*}
\rk(\cV)&=\rk(\cV_1)+\rk(\cV_2),\\
\irreg(\cV)&=\irreg(\cV_1)+\irreg(\cV_2),\\
\dim H^0_{dR}(K_x,\cV)&\le\dim H^0_{dR}(K_x,\cV_1)+\dim H^0_{dR}(K_x,\cV_2).
\end{align*}
\end{proof}

\begin{corollary} For any $\cV\in\DHol{K_x}$, there is irreducible $\cV'\in\DHol{K_x}$ such that
$$\delta(\END(\cV))\ge\frac{\rk(\cV)}{\rk(\cV')}\delta(\HOM(\cV',\cV)).$$
\label{co:maxslope}
\end{corollary}
\begin{proof}
Let $\cV_1,\dots, \cV_k$ be the irreducible components of $\cV$ (with multiplicity). Take $\cV'$ to be the $\cV_i$ that minimizes
\begin{equation}
\frac{\delta(\HOM(\cV',\cV))}{\rk{\cV'}}. \label{eq:deltatensor}
\end{equation}
Semi-additivity of $\delta$ implies that
$$\delta(\END(\cV))\ge\sum_i\delta(\HOM(\cV_i,\cV)),$$
and Corollary~\ref{co:maxslope} follows. 
\end{proof}

\begin{remark*} It is easy to see that this choice of $\cV'$ actually minimizes 
\eqref{eq:deltatensor}
over all $\cV'\in\DHol{K_x}$. 
Also, if irreducible $\cV'\in\DHol{K_x}$ minimizes \eqref{eq:deltatensor}, then 
$\cV'$ is a component of $\cV$.
\end{remark*}

Now let $\cL$ be as in Theorem \ref{th:maintheorem}; for every $x\in\p1-U$, Corollary \ref{co:maxslope} yields 
an irreducible object $\cV_x\in\DHol{K_x}$ such that
$$\delta(\END(\Psi_x(\cL)))\ge\frac{\rk(\cL)}{\rk{\cV_x}}\delta(\HOM(\cV_x,\Psi_x(\cL))).$$ It remains
to prove the following statements:

\begin{proposition}[Theorem~\ref{th:maintheorem}, Case Ia] Suppose there is a rank one local system $\ell$ on $U$ such that
$\cV_x=\Psi_x(\ell)$ for every $x\in\p1-U$. Then either $\Hom(\ell,\cL)$ or $\Hom(\cL,\ell)$ is non-zero.
\label{pp:Case1a}
\end{proposition}
\begin{proof}
It suffices to show that $\chi_{dR}(j_{!*}(\HOM(\ell,\cL)))>0$. Indeed, by Proposition~\ref{pp:EulerPoincare}, we have
\begin{equation*}
\begin{split}
\chi_{dR}(j_{!*}(\HOM(\ell,\cL)))=&2\rk(\cL)-\sum\delta(\HOM(\cV_x,\Psi_x(\cL)))\\ \ge&
2\rk(\cL)-\frac{1}{\rk(\cL)}\sum\delta(\END(\Psi_x(\cL)))\\=&\frac{\rig(\cL)}{\rk{\cL}}=\frac{2}{\rk(\cL)}>0.
\end{split}
\end{equation*}
\end{proof}

\begin{proposition}[Theorem~\ref{th:maintheorem}, Case Ib] Suppose there is $\lambda\in\kk-\Z$ and a rank one connection $\ell$ on $U$ satisfying \eqref{eq:twistedell}.
Then
$\rk(\HOM(\ell,\cL)\star_{mid}\cK^\lambda)<\rk(\cL)$.
\label{pp:Case1b}
\end{proposition}
\begin{proof}
By Proposition~\ref{pp:rankradon}, we have
\begin{equation*}
\begin{split}
\rk(\HOM(\ell,\cL)\star_{mid}\cK^\lambda)=&\sum_{x\in\p1-U}\delta(\HOM(\cV_x,\Psi_x(\cL)))-\rk(\cL)\\ \le&
\frac{1}{\rk(\cL)}\sum\delta(\END(\Psi_x(\cL)))-\rk(\cL)\\=&
\rk(\cL)-\frac{\rig(\cL)}{\rk(\cL)}=\rk(\cL)-\frac{2}{\rk(\cL)}<\rk(\cL).
\end{split}
\end{equation*}
\end{proof}

\subsection{Details of the proof: Case II}

Again, we start with some local results at fixed $x\in\p1$.

\begin{lemma}\label{lm:irrslope} Suppose $\cV,\cW\in\DHol{K_x}$ are irreducible.
\begin{enumerate} 
\item \label{it:irrslope1} If $\slope(\cV)\ne\slope(\cW)$, then 
\begin{align*}
\slope(\HOM(\cV,\cW))&=\max(\slope(\cV),\slope(\cW)),\\
\frac{\delta(\HOM(\cV,\cW))}{\rk(\cV)\rk(\cW)}&=1+\max(\slope(\cV),\slope(\cW)).
\end{align*}

\item \label{it:irrslope2} If $\slope(\cV)=\slope(\cW)$ has denominator $d$, then
\begin{align*}
\slope(\HOM(\cV,\cW))&\ge\left(1-\frac{1}{d}\right)\slope(\cW),\\
\frac{\delta(\HOM(\cV,\cW))}{\rk(\cV)\rk(\cW)}&\ge
1-\frac{1}{d^2}+\left(1-\frac{1}{d}\right)\slope(\cW).
\end{align*}
\end{enumerate} 
\end{lemma}
\begin{proof}
We use classification of connections on formal disk (see for instance \cite[Theorem III.1.2]{Mal}). 
There exists a ramified extension $$\tK_x=\kk((z^{1/r}))\supset K_x=\kk((z))$$ and 
isomorphisms
\begin{align*}
\cV\otimes\tK_x&\simeq(\tK_x^{\rk\cV},d+\diag(\mu_\cV^{(1)},\dots,\mu_\cV^{(\rk\cV)})),&
\mu_\cV^{(i)}&\in\Omega_{\tK_x}=\kk((z^{1/r}))dz;\\
\cW\otimes\tK_x&\simeq(\tK_x^{\rk\cW},d+\diag(\mu_\cW^{(1)},\dots,\mu_\cW^{(\rk\cW)})),&
\mu_\cW^{(j)}&\in\Omega_{\tK_x}.
\end{align*}
Let us denote by $\ord(\mu)$ the order of $\mu\in\Omega_\tK$ in $z$ (which might be fractional). Then 
\begin{align*}
\ord(\mu_\cV^{(i)})&=-1-\slope(\cV) &(i&=1,\dots,\rk(\cV)),\\
\ord(\mu_\cW^{(j)})&=-1-\slope(\cW) &(j&=1,\dots,\rk(\cW)).
\end{align*}
Then
\begin{equation}
\irreg(\HOM(\cV,\cW))=\sum_{i,j}\max(-1-\ord(\mu_\cW^{(j)}-\mu_\cV^{(i)}),0)
\label{eq:tensor}
\end{equation}

Proof of \eqref{it:irrslope1}. The leading terms of $\mu_\cW^{(j)}-\mu_\cV^{(i)}$ do not cancel, so 
$$\ord(\mu_\cW^{(j)}-\mu_\cV^{(i)})=\min(\ord(\mu_\cW^{(j)}),\ord(\mu_\cV^{(i)})).$$
Now \eqref{eq:tensor} implies the formula for $\slope(\HOM(\cV,\cW))$. To prove the formula for $\delta(\HOM(\cV,\cW))$, we notice
that $H^0(K_x,\HOM(\cV,\cW))=\Hom(\cV,\cW)=0$, so $$\delta(\HOM(\cV,\cW))=\irreg(\HOM(\cV,\cW))+\rk(\HOM(\cV,\cW)).$$

Proof of \eqref{it:irrslope2}. Now cancellation in the leading terms of $\mu_\cW^{(j)}-\mu_\cV^{(i)}$ is possible. However,
$\cV\otimes\tK_x$ carries an action of the Galois group $\Gal(\tK_x/K_x)$. In particular, leading terms of $\mu_\cV^{(j)}$ come in $d$-tuples of the form
$$\{\zeta a z^{-1-\slope(\cV)}dz:\zeta\in\kk,\zeta^d=1\}$$
for some fixed $a\in\kk-\{0\}$. Therefore, among the differences $\mu_\cW^{(j)}-\mu_\cV^{(i)}$,
not more than one out of every $d$ has cancellation.  Now \eqref{eq:tensor} implies the formula for $\slope(\HOM(\cV,\cW))$. Finally, 
$\dim\Hom(\cV,\cW)\le 1$, and so
\begin{equation*}
\begin{split}
\frac{\delta(\HOM(\cV,\cW))}{\rk(\cV)\rk(\cW)}&\ge\frac{\rk(\cV)\rk(\cW)-1}{\rk(\cV)\rk(\cW)}+
\slope(\HOM(\cV,\cW))\\
&\ge 1-\frac{1}{\rk(\cV)\rk(\cW)}+\left(1-\frac{1}{d}\right)\slope(\cW).
\end{split}
\end{equation*}
It remains to notice that $\rk(\cV),\rk(\cW)\ge d$.
\end{proof}
\begin{lemma}\label{lm:slope} Suppose $\cV,\cW\in\DHol{K_x}$, and $\cV$ is irreducible.
\begin{enumerate}
\item\label{it:slope1} If $\rk(\cV)>1$, then $\delta(\HOM(\cV,\cW))\ge\rk(\cV)\rk(\cW)$.
\item\label{it:slope2} If $\slope(\cV)>2$ is not an integer, then $\delta(\HOM(\cV,\cW))\ge 2\rk(\cV)\rk(\cW)$.
\end{enumerate}
\end{lemma}
\begin{proof}
By semiadditivity of $\delta$, we may assume that $\cW$ is irreducible without losing generality.

\eqref{it:slope1} Without loss of generality, we may assume that $\slope(\cV)$ is not an integer. Indeed, we can replace $\cV$ and $\cW$ with $\HOM(\ell,\cV)$ and $\HOM(\ell,\cW)$
for any $\ell\in\DHol{K}$ of rank one, and we can choose $\ell$ so that $\slope(\HOM(\ell,\cV))$ is not an integer (as in Remark~\ref{rm:approximation}). 
Now the statement follows from Lemma~\ref{lm:irrslope}\eqref{it:irrslope1} if $\slope(\cV)\ne\slope(\cW)$ or from Lemma~\ref{lm:irrslope}\eqref{it:irrslope2} if $\slope(\cV)=\slope(\cW)$.

\eqref{it:slope2} If $\slope(\cV)\ne\slope(\cW)$, the statement follows from Lemma~\ref{lm:irrslope}\eqref{it:irrslope1}. Assume $\slope(\cV)=\slope(\cW)$. Then 
$\slope(\cW)\ge 2+\frac{1}{d}$, where $d$ is the denominator of $\slope(\cW)$. Lemma~\ref{lm:irrslope}\eqref{it:irrslope2} implies
$$\frac{\delta(\HOM(\cV,\cW))}{\rk(\cV)\rk(\cW)}\ge 1-\frac{1}{d^2}+\left(1-\frac{1}{d}\right)\cdot\left(2+\frac{1}{d}\right)
=2+\frac{d^2-d-2}{d^2}\ge2.$$
\end{proof}

\begin{lemma}\label{lm:slope2} Suppose $\cV\in\DHol{K_x}$ is irreducible and $\slope(\cV)<2$ is not
an integer. Then for any $\cW\in\DHol{K_x}$, 
$$\delta(\HOM(\cV,\cW))
\ge(\irreg(\cW^{>1})-\rk(\cW^{>1}))\rk(\cV)+\rk(\cV)\rk(\cW).$$
(Recall that $\cW^{>1}$ is the maximal submodule of $\cW$ whose components all have slopes greater than one.)
\end{lemma}
\begin{proof}
By semiadditivity of $\delta$, we can assume that $\cW$ is irreducible (the right-hand side is additive in $\cW$). The statement follows from Lemma~\ref{lm:slope}\eqref{it:slope1} if
$\slope(\cW)\le 1$, so assume $\slope(\cW)>1$. Then $\cW=\cW^{>1}$, and we have to show that
$$\frac{\delta(\HOM(\cV,\cW))}{\rk(\cV)\rk(\cW)}\ge\slope(\cW).$$ If $\slope(\cV)\ne\slope(\cW)$, this follows from Lemma~\ref{lm:irrslope}\eqref{it:irrslope1}. Suppose therefore that $\slope(\cV)=\slope(\cW)$. 
Then by Lemma~\ref{lm:irrslope}\eqref{it:irrslope2}, we have
$$\frac{\delta(\HOM(\cV,\cW))}{\rk(\cV)\rk(\cW)}-\slope(\cW)\ge 1-\frac{1}{d^2}-\frac{\slope(\cW)}{d}\ge
1-\frac{1}{d^2}-\frac{1}{d}\left(2-\frac{1}{d}\right)=1-\frac{2}{d}\ge 0.$$
Here we used that $\slope(\cV)=\slope(\cW)\le 2-\frac{1}{d}$.
\end{proof}

Let us return to the second case of Theorem~\ref{th:maintheorem}. We need to verify the following two claims.

\begin{proposition} There is at most one $x\in\p1-U$ such that $\rk(\cV_x)>1$.
\label{pp:therecanbeonlyone}
\end{proposition}
\begin{proof}
Indeed, by Lemma~\ref{lm:slope}\eqref{it:slope1}, whenever $\rk(\cV_x)>1$, we have
$$\delta(\END(\Psi_x(\cL)))\ge\frac{\rk(\cL)}{\rk(\cV_x)}\delta(\HOM(\cV_x,\Psi_x(\cL)))
\ge\rk(\cL)^2.$$
But by rigidity,
$$\sum_x\delta(\END(\Psi_x(\cL)))=2\rk(\cL)^2-\rig(\cL)<2\rk(\cL)^2.$$
\end{proof}

\begin{proposition}[Theorem~\ref{th:maintheorem}, Case II] \label{pp:Case2}
Suppose for $\infty\in\p1-U$, $\rk(\cV_\infty)>1$. Choose a rank one
connection $\ell$ on $U$ such that $\cV_x\simeq\Psi_x(\ell)$ for $x\in\A1-U$, and
the slope of $\HOM(\Psi_\infty(\ell),\cV_\infty)$ is not an integer. Then
$\rk(\ft{\HOM(\ell,\cL)})<\rk(\cL)$.
\end{proposition}
\begin{proof}
Indeed, by Proposition \ref{pp:rankfourier}, we have
\begin{multline*}
\rk(\ft{\HOM(\ell,\cL)})=\sum_{x\in\A1-U}\delta(\HOM(\cV_x,\Psi_x(\cL)))\\
+\irreg(\Psi_\infty(\HOM(\ell,\cL))^{>1})-\rk\Psi_\infty(\HOM(\ell,\cL))^{>1}.
\end{multline*}
It suffices to prove the inequality
\begin{multline}
\irreg(\Psi_\infty(\HOM(\ell,\cL))^{>1})-\rk\Psi_\infty(\HOM(\ell,\cL))^{>1}\\ \le
\frac{\delta(\HOM(\cV_\infty,\Psi_\infty(\cL)))}{\rk(\cV_\infty)}-\rk(\cL),
\label{eq:needed}
\end{multline}
and then use the argument of Proposition \ref{pp:Case1b}.

To prove \eqref{eq:needed}, take $$\cV=\HOM(\Psi_\infty(\ell),\cV_\infty),\qquad \cW=\HOM(\Psi_\infty(\ell),\Psi_\infty(\cL)).$$ 
By the argument used to prove Proposition~\ref{pp:therecanbeonlyone},
$\delta(\HOM(\cV,\cW))< 2\rk(\cV)\rk(\cW)$. We then see that $\slope(\cV)<2$ by 
Lemma~\ref{lm:slope}\eqref{it:slope2}. 
Finally, \eqref{eq:needed} follows from Lemma~\ref{lm:slope2}.
\end{proof}

\section{Applications}
\subsection{Irregular Deligne-Simpson problem} \label{sc:DS}
Irregular Katz's algorithm can be applied to the `irregular Deligne-Simpson problem' for rigid local systems. In the case of regular singularities, this is explained in
\cite[Section~6.4]{Ka}, and the irregular case is quite similar. 

\begin{definition} A \emph{formal type datum} is a collection of isomorphism classes
 $$\{[\cV_x]\}_{x\in\p1}$$ of connections $\cV_x\in\DHol{K_x}$ such that the following conditions hold:
 \begin{enumerate}
 \item $r=\rk\cV_x$ does not depend on $x$;
 \item For all but finitely many $x$, $\cV_x$ is trivial: $\cV_x\simeq(K_x^r,d)$;
 \item $\sum_x\res(\bigwedge^r\cV_x)\in\Z$. (Since $\bigwedge^r\cV_x\in\DHol{K_x}$ has 
rank one, its residue makes sense as an element of $\kk/\Z$.)
 \end{enumerate}

A solution of the  \emph{(irregular) Deligne-Simpson problem} corresponding to $\{[\cV_x]\}$ is
an irreducible connection $\cL$ on an open subset of $\p1$ with prescribed formal type: $\Psi_x(\cL)\simeq\cV_x$
for all $x$.

The \emph{rigidity index} of $\{[\cV_x]\}$ is
$$\rig\{[\cV_x]\}=2r^2-\sum_{x\in\p1}\delta(\END(\cV_x)),$$
where $\delta$ is defined by \eqref{eq:delta}. 
\end{definition}

Suppose that $\cL$ solves the Deligne-Simpson problem for $\{[\cV_x]\}$. By Proposition~\ref{pp:EulerPoincare}, $\rig(\cL)=\rig\{[\cV_x]\}$. In particular, $\rig\{[\cV_x]\}\le 2$. 

Let $\ft\cL$ be the Fourier transform of $\cL$, and let $\{[\ft\cV_x]\}$ be the formal type of $\ft\cL$: $\ft\cV_x=\Psi_x(\ft\cL)$. One can check that $\{[\ft\cV_x]\}$ is determined by $\{[\cV_x]\}$; essentially, $\ft\cV_x$ is given by the local Fourier transform of \cite{BE} (this is discussed in more details in \cite{LF}). In other words, we obtain a notion of the Fourier transform for formal type data, and $\{[\ft\cV_x]\}$ is the Fourier transform of
$\{[\cV_x]\}$.

For arbitrary formal type datum $\{[\cV_x]\}$, its Fourier transform $\{[\ft\cV_x]\}$ might be undefined. Actually, $[\ft\cV_x]$ (for $x\ne\infty$)
is constructed from $\{[\cV_x]\}$ in two steps: the local Fourier transform describes the
quotient $\ft\cV_x/(\ft\cV_x)^{hor}$ modulo the maximal trivial subconnection, while Proposition~\ref{pp:rankfourier} gives a formula for $\dim(\ft\cV_x)$. This determines the
isomorphism class $[\ft\cV_x]$, assuming the obvious compatibility condition
$\dim(\ft\cV_x/(\ft\cV_x)^{hor})\le\dim\ft\cV_x$. If the compatibility condition fails, the
Fourier transform of $\{[\cV_x]\}$ is undefined. 

The Fourier transform $\cL\to\ft\cL$ provides a one-to-one correspondence between solutions
to the Deligne-Simpson problems for $\{[\cV_x]\}$ and $\{[\ft\cV_x]\}$. If $\{[\ft\cV_x]\}$ is undefined, the Deligne-Simpson problem for $\{[\cV_x]\}$ has no solutions.

The situation for middle convolution $\cL\star_{mid}\cK^\lambda$ is similar to that for the Fourier transform. Again, it makes sense for formal type data, but it is not always defined.
If the formal type data are related by the middle convolution, their Deligne-Simpson problems
are equivalent. If the middle convolution of a formal type datum is undefined, its Deligne-Simpson problem has no solutions.

Now let us analyze the Deligne-Simpson problem for
a formal type datum $\{[\cV_x]\}$ in the case $\rig\{[\cV_x]\}=2$. We can run irregular
Katz's algorithm on the level of formal type data. On each step, we decrease the
rank of the formal type datum using either the middle convolution or the Fourier transform,
assuming that they are defined. After finitely many steps, we arrive at one of the two situations:

\begin{itemize}
\item Irregular Katz's algorithm decreases the rank of the formal type datum to one.
Then the Deligne-Simpson problem for $\{[\cV_x]\}$ is equivalent to the Deligne-Simpson problem for a formal type datum of rank one, which is clearly solvable.

\item The output of a step of irregular Katz's algorithm is undefined, and then
the Deligne-Simpson problem for $\{[\cV_x]\}$ has no solutions.
\end{itemize}
 
\begin{remark*}
In \cite{Si}, C.~Simpson uses Katz's algorithm to analyze the (regular)
Deligne-Simpson problem without restrictions on the rigidity index. We do not know whether
irregular Katz's algorithm can be used for similar analysis in the irregular case.
\end{remark*}

\subsection{Rigidity index zero and Lax pairs for Painlev\'e equations} \label{sc:rig0}

Irregular Katz's algorithm can be also used to classify connections $\cL$ of rigidity index $0$; the details will be given elsewhere. 
In the case of regular singularities, such classification was proved by Kostov 
\cite[Lemma~17]{Ko}. 

We claim that the proof of Theorem~\ref{th:maintheorem} can be modified for such $\cL$. It
is not true that the rank of $\cL$ can always be decreased to $1$, however, the algorithm's
\emph{stopping points} (that is, the connections whose rank cannot be decreased) can be 
described.

Let us study the moduli spaces of connections. For a fixed 
formal type datum $\{[\cV_x]\}$, consider the moduli space $\fM=\fM_{\{[\cV_x]\}}$
of irreducible connections of this formal type. Equivalently, points of $\fM$
are solutions to the Deligne-Simpson problem. Then $$\dim\fM=
-2-\rig(\{[\cV_x]\});$$ 
in particular if $\rig(\{\cV_x\})=0$, $\fM$ is a surface.

If $\{[\ft\cV_x]\}$ is the Fourier transform of $\{[\cV_x]\}$, we get an isomorphism
$$\fM_{\{[\cV_x]\}}\iso\fM_{\{[\ft\cV_x]\}}:\cL\mapsto\ft\cL.$$
Similarly, middle convolution $\cL\mapsto\cL\star_{mid}\cK^\lambda$ induces an isomorphism between moduli spaces. Therefore, the space $\fM$ does not change as we apply  
irregular Katz's algorithm to $\{[\cV_x]\}$. In this way, we can always reduce to the case
when the formal type $\cV_x$ is a stopping point of the algorithm.

Important examples are spaces $\fM$ for $\rk\cV_x=2$. Assume $\rig\{[\cV_x]\}=0$, so $\dim(\fM([\cV_x]))=2$. Note that $\{[\cV_x]\}$ is automatically a stopping point of
the algorithm, because its rank cannot be decreased to one, as all rank one systems are rigid.
The surface $\fM$ is the space of initial conditions
of a Painlev\'e equation $P_*$, where index $*=I,II,\dots,VI$ depends on $\{[\cV_x]\}$.
Geometrically, $P_*$ controls the isomonodromy deformation of connections. The isomorphisms induced by the Fourier transform and the middle convolution respect the isomonodromy deformations. Therefore, if generalized Katz's algorithm reduces formal type datum $\{[\cW_x]\}$ to $\{[\cV_x]\}$, the isomonodromy deformation of connections of type $\{[\cW_x]\}$
is also given by $P_*$. In other words, $\{[\cW_x]\}$ gives another Lax pair for $P_*$. In this manner, irregular Katz's algorithm in case of rigidity index zero can be viewed as a 
reduction algorithm for Lax pairs for Painlev\'e equations.

\begin{remark*} Only the sixth Painlev\'e equation $P_{VI}$ appears in the classification of
\cite[Lemma~17]{Ko}; the other Painlev\'e equations correspond to irregular formal types.
\end{remark*}

\section{Remarks}
\subsection{Middle convolution via twisted differential operators} \label{sc:mid}
The middle convolution with Kummer local system is naturally formulated in terms of rings of twisted differential operators (or TDOs). 

Denote by $\cD_1$ the TDO ring acting on $\cO_\p1(1)$ (see \cite{BB2} for the definition of
TDO ring). Let us `scale' $\cD_1$ by a fixed number $\lambda\in\kk$, denote the
resulting TDO by $\cD_\lambda$. Informally, $\cD_\lambda$ is the ring of differential operators
on $\cO_\p1(\lambda)$.

\begin{remark*} Consider the natural projection $p:\A2-\{0\}\to\p1$.
We can interpret holonomic $\cD_\lambda$-modules as $\cD$-modules $M$ on $\A2-\{0\}$ such
that the restriction of $M$ to any fiber $p^{-1}(x)$ is a sum of several copies of $\cK^{\lambda}$. Informally, we require that $M$ is a \emph{monodromic} $\cD$-module whose
restriction to each fiber has `monodromy $e^{2\pi i \lambda}$'. 
\end{remark*}

Suppose that $\lambda\in\kk-\Z$. In \cite{AE}, A.~D'Agnolo and M.~Eastwood present an
equivalence $\Rad$ (\emph{the Radon transform}) between the category
of $\D_\lambda$-modules and that of $\D_{-\lambda}$-modules. (One should keep in mind that up to equivalence,
the category of $\D_\lambda$-modules depends only on the image of $\lambda$ in $\kk/\Z$.)
$\Rad$ can be viewed as a twisted version of the transform defined by J.~-L.~Brylinski in \cite{Br}.
In a sense, it is also a particular case of the Radon transform defined by A.~Braverman and A.~Polishchuk, who consider monodromic sheaves whose monodromy need not be scalar 
(\cite{BP}).

Explicitly, $\Rad$ can be defined as the 
integral transform whose kernel is a rank one $\D_\lambda\boxtimes\D_\lambda$-module 
on $\p1\times\p1$ with a simple pole along the diagonal (and no other singularities). 
Alternatively, if one interprets $\D_\lambda$-modules as monodromic $\D$-modules on $\A2-\{0\}$, the equivalence is simply the Fourier transform on $\A2$.

We can view the middle convolution with Kummer local system as a composition of the Goresky-MacPherson extension and the Radon transform as follows.
A connection $\cL$ on an open set $U\subset\A1$ can be viewed as a $\D_\lambda|_U$-module
using the trivialization $\D_\lambda|_\A1=\D_\A1$. We then extend it to a 
$\D_\lambda$-module $j_{!*}\cL$ for $j:U\hookrightarrow\p1$. The Radon transform
$\Rad(j_{!*}\cL)$ is a holonomic $\D_{-\lambda}$-module that is smooth on $U$. Its restriction
to $U$ is a connection which equals $\cL\star_{mid}\cK^\lambda$.  

The first case of Theorem~\ref{th:maintheorem} can be reformulated:
\begin{enumerate}
\item[(\ref{cs:R}')] \sl{There is a rank one $\D_{-\lambda}|_U$-module $\ell$ such that
$$\rk\Rad(j_{!*}\HOM(\ell,\cL))<\rk(\cL).$$}
\end{enumerate}

Note that the point $\infty$ plays no special role in this formulation. Similarly, if one rewrites Proposition~\ref{pp:rankradon} using the Radon transform, 
the special treatment of $\infty$ is not necessary, essentially because it plays no special role in definition of $\Rad$. 

\begin{remark*} In \cite{Si}, C.~Simpson studies the middle convolution via the so-called `convoluters'. One can rewrite Theorem~\ref{th:maintheorem}\eqref{cs:R} using the
de Rham version of convoluters (\cite[Section~3.3]{Si}) with irregular singularities.
From the viewpoint of TDO rings, the convoluter encodes the $\D_{-\lambda}$-module
$\ell$ from $(\ref{cs:R}')$. 
\end{remark*}

\subsection{$l$-adic version of irregular Katz's algorithm}
The following observation is due to P.~Deligne. 

Most of the proof of Theorem~\ref{th:maintheorem} remains valid in the settings of $l$-adic
sheaves. The only exception is Lemma~\ref{lm:irrslope}. Its first statement still holds
(see \cite[Lemma~1.3]{Ka2}), but the second statement requires the additional assumption that $d$ (the denominator of the slope) is not divisible by the characteristic of the ground field. 
Let us make the statement precise.

Let $K$ be the fraction field of a Henselian valuation ring whose residue field is perfect of finite characteristic $p$. Denote by $I\subset\Gal(K^{sep}/K)$ the inertia
group of $K$ and by $P\subset I$ its Sylow's $p$-group. For a continuous finite-dimensional
representation $V$ of $P$ (over a fixed $l$-adic field), we denote its break decomposition by by $$V=\bigoplus_{s\in\Q} V(s).$$ One can check the following statement.

\begin{lemma}
Let $V$ and $W$ be continuous finite-dimensional
representations of $\Gal(K^{sep}/K)$. Fix $s\in\Q$ with denominator
$d$, and suppose $p$ does not divide $d$. Then 
\begin{equation}
\dim(\Hom(V,W)(s))\ge\dim V(s)\dim W(s)\left(1-\frac{1}{d}\right). \label{eq:ladic}
\end{equation}
\qed
\label{lm:ladic}
\end{lemma}

\begin{remark*} Since Lemma~\ref{lm:ladic} holds for all $V$, $W$, one can replace them by $V(s)$ and $W(s)$. In other words, in \eqref{eq:ladic} one can replace 
$\Hom(V,W)(s)$ with the image of $\Hom(V(s),W(s))$ in this
space.
\end{remark*}

In particular, \eqref{eq:ladic} holds if either $\dim(V)$ or $\dim(W)$ is less than $p$.
This implies that the extension of Katz's algorithm works for wild $l$-adic local systems
whose rank does not exceed the characteristic $p$ of the ground field.

\nocite{*}
\bibliographystyle{amsalpha}
\bibliography{rigidrefs}
\end{document}